\documentclass[english, 11pt]{amsart}

\usepackage{tikz}
\usepackage{aliascnt}
\usepackage{amsfonts}
\usepackage{mathrsfs}
\usepackage{amsthm}
\usepackage{amsmath,amssymb}
\usepackage[all,poly,knot]{xy}
\usepackage{amsfonts}
\usepackage{color}
\usepackage{hyperref}
\usepackage{comment}  
\usepackage{csquotes}   

\newcommand{\lowerromannumeral}[1]{\romannumeral#1\relax}

\theoremstyle{plain}
\newtheorem{thmx}{Theorem} 

\newtheorem{thm}{Theorem}[section]  

\newtheorem{cor}[thm]{{Corollary}} 
\newtheorem{corx}[thmx]{Corollary}
\newtheorem{lem}[thm]{{Lemma}}

\newtheorem{prop}[thm]{Proposition}

\newtheorem{conj}[thm]{{Conjecture}}
\newtheorem{ques}[thm]{{Question}}

\newtheorem{defi}[thm]{Definition}
\newtheorem*{claim}{Claim}

\theoremstyle{remark}
\newtheorem{rmk}[thm]{Remark}

\numberwithin{equation}{section}

\def\C{\mathbb C}

\def\log{\mathrm{log}\,}
\def\d{\mathrm{d}}

\def\ep{\varepsilon} 
\def\hess{\mathrm{d}\mathrm{d}^c}
\def\sn{\sqrt{-1}}
\def\bD{\mathbb{D}}

\usepackage[top=1.5in, bottom=1.5in, left=1in, right=1in]{geometry}

 \usepackage[hyperpageref]{backref}
 
\usepackage{xcolor}
\hypersetup{
	colorlinks,
	linkcolor={red!50!black},
	citecolor={blue!50!black},
	urlcolor={blue!80!black}
}
\begin{document} 
\title[Picard theorems for moduli of polarized varieties]{Picard theorems for moduli spaces of polarized varieties} 
	\author{Ya Deng}
	 \address{
		Institut des Hautes \'Etudes Scientifiques,
		Universit\'e  Paris-Saclay, 35 route de Chartres, 91440, Bures-sur-Yvette, France}
	\email{deng@ihes.fr}
	 
	\author{Steven Lu}
		 \address{D\'epartement de math\'ematiques
		 	Universit\'edu Qu\'ebec \`a Montr\'eal
		 	Case postale 8888, succursale centre-ville
		 	Montréal (Québec) H3C 3P8}
	\email{lu.steven@uqam.ca}
	\author{Ruiran Sun}
		 \address{Institut fur Mathematik, Universit\"at Mainz, Mainz, 55099, Germany}
	\email{ruirasun@uni-mainz.de}
        \author{Kang Zuo}
        	 \address{Institut fur Mathematik, Universit\"at Mainz, Mainz, 55099, Germany}
        \email{zuok@uni-mainz.de}

\begin{abstract} {As a result of our study of the hyperbolicity of the moduli space of polarized manifold,
we give a general big Picard theorem for a holomorphic curve on a log-smooth pair $(X,D)$ such that $W=X\setminus D$ admits a Finsler pseudometric that is strongly negatively curved when pulled back to the curve. We show, by some refinements of the classical Viehweg-Zuo construction, that this latter condition holds for the base space $W$, if nonsingular, of any algebraic family of polarized complex projective manifolds with semi-ample canonical bundles whose induced moduli map $\phi$ to the moduli space of such manifolds is generically finite and any $\phi$-horizontal holomorphic curve in $W$. This yields the big Picard theorem for any holomorphic curves in the base space $U$ of such an algebraic family by allowing this base space to be singular but with generically finite moduli map. An immediate and useful corollary is that any holomorphic map from an algebraic variety to such a base space $U$ must be algebraic, i.e., the corresponding holomorphic family must be algebraic. We also show the related algebraic hyperbolicity property of such a base space $U$, {which generalizes previous Arakelov inequalities and weak boundedness results for moduli stacks and offers, in addition to the Picard theorem above, another evidence in favor of the hyperbolic embeddability of such an $U$.}}
\end{abstract}


\maketitle
\tableofcontents 
\section{Introduction}
{The paper of Viehweg and Zuo \cite{VZ-1} pioneered the investigation of  the analytic hyperbolicity of the moduli space of polarized varieties by showing that, for any algebraic family of canonically polarized complex manifolds with quasi-finite moduli map to the moduli space, the base space $U$ of the family is Brody hyperbolic, i.e., $U$ admits no non-constant holomorphic maps from $\mathbb C$.\footnote{The term ``moduli space'' used above should more properly be written as ``moduli stack'' though we will not belabour this point in this paper. Hence hyperbolicity of moduli space is taken to mean in the sense of stacks. } \\[0.5mm] 
\indent
For compact spaces, Brody hyperbolicity is equivalent by Brody's reparametrization theorem to being hyperbolic in the sense of Kobayashi. However, most parameter spaces considered above are not compact. In the non-compact setting, the strongest notion of complex hyperbolicity (short of the existence of a hyperbolic compactification) is that of hyperbolic embeddability, meaning that there exists a compactification in which the space hyperbolically embeds.\footnote{Hyperbolic embedding characterizes the Bailey-Borel compactifications of the moduli spaces of abelian varieties, a classical theorem of Borel. Since we are mainly interested in the more algebraic aspects of hyperbolicity in this paper, we do not give the definition of Kobayashi hyperbolicity, nor of hyperbolic embedding. The interested reader should consult standard references such as the books of Lang and of Kobayashi on complex hyperbolic geometry.}\\[0.5mm] 
\indent
This paper offers two evidences of this hyperbolic embeddability in the moduli setting: we obtain for the said parameter spaces that (I) the Picard extension property of holomorphic maps from curves to them holds, generalizing the Brody hyperbolicity result of \cite{VZ-1}, and that (II) they are algebraically hyperbolic in the sense of Demailly, see Definition~\ref{def:DC}. In fact, we generalize these results, after a suitable recall and recasting of the construction in the above paper of Viehweg and Zuo, from the setting of canonically polarized manifolds to the setting of manifolds with semi-ample canonical bundle. 
It is fitting to remark here that hyperbolic embeddability generalizes Kobayashi hyperbolicity and that property (I) and (II) are implied by it though the converses corresponding to (I) and (II) are open, see e.g. \cite[II\S 2]{Lang}, \cite[Theorem~(6.3.7)]{Kobayashi} and \cite{PR07}. 
}\\[-3.7mm]

More precisely, {consider an algebraic family of polarized complex projective manifolds with semi-ample canonical divisors and Hilbert polynomial $h$ given by {an equivalence class of pairs} $(f:\, V \to U,\mathcal{L})\in \mathcal{M}_h(U)$, i.e. $f$ is a proper and smooth morphism between algebraic varieties $V$ and $U$ with a relatively ample line bundle $\mathcal{L}$ over $V$ having $h$ as the fibrewise Hilbert polynomial {(modulo the obvious equivalence)}}. Suppose that the induced map from the parameter space $U$ to the coarse moduli space $M_h$, {also called the} moduli map, is quasi-finite. Let $\bar U$ be an algebraic  compactification of $U$, {$C\subset \bar C$ an open subset of a complex curve with codimension one complement and $\gamma: C \to U$ a holomorphic map.  We show that $\gamma$ has a holomorphic extension $\bar\gamma: \bar C \to \bar U$ and, as a well-known consequence, that {the result generalizes} to the case when $C$ is replaced by a complex space and holomorphic by  the word meromorphic, obtaining in particular that $\gamma$ is algebraic when $C$ is an algebraic variety.} Our proof of this proceeds in two major steps, comprising the two major parts of substance of the paper, the first extracts the necessary curvature conditions via a detailed recall {and some necessary extensions} of the Viehweg-Zuo paper \cite{VZ-1} while the second uses Nevanlinna theory to prove our key technical theorem of independent interest, {Theorem~\ref{thm:Big Picard}.} 
We offer two proofs of this technical theorem. {The first offers a simplification and strengthening of some techniques developed by and used in Griffiths-King \cite{Gri-King-73} to generalize a classical big Picard theorem in the setting of moduli spaces and follows a metric approach to Nevanlinna theory as espoused by Chern and  Ahlfors.  The second is inspired by the proof of the fundamental vanishing theorem of Siu-Yeung and Demailly (\cite{S-Y97, Dem97b}) for jet-differentials on holomorphic curves and proceeds by a reduction to the  logarithmic derivative lemma. {As a natural consequence of the first part of the paper},} we show in the last section the algebraically hyperbolicity of $U$,  generalizing previous Arakelov inequalities and weak boundedness results of moduli stacks (see Remark \ref{rem:Arakelov}). 
\\[-3.5mm]

{The Shafarevich problem from the 60's and its higher dimensional generalizations (cf.\ \cite{Vieh-01,Kov-03} for an introduction), motivated the very extensive modern study of the hyperbolicity properties of the moduli space of polarized manifolds with many major advances. We will only touch upon a couple of the recent avances that generalize the result of \cite{VZ-1} below.}\\[-3.5mm]
%

{By refining a part of Viehweg-Zuo's construction of negatively curved pseudometrics,  Popa-Taji-Wu in \cite{PTW-18} proved the Brody hyperbolicity for the case of the moduli space of polarized manifold with nef and big canonical bundle.} {In \cite{To-Yeung-15} and \cite{To-Yeung-18}, To and Yeung proved that the moduli spaces of canonically polarized and  Ricci flat manifolds are Kobayashi hyperbolic by constructing negatively curved Finsler metrics on $U$. This is stronger than Brody hyperbolicity for a noncompact $U$, but still weaker than having a hyperbolic embedding (cf. \cite[Chapter~2]{Lang} or {\cite[Chapter~3, \S3]{Kobayashi}}). More recently, the first named author in \cite{Deng-18b} proved in addition that the moduli space of polarized varieties whose canonical divisors are semi ample and big is Kobayashi hyperbolic and that the moduli space without the bigness condition is Brody hyperbolic, both of which were conjectured by Viehweg and the fourth named author (\cite[Question 0.2]{VZ-1}). This was done via negatively curved Finsler metrics on $U$ 
obtained by adding, at the end of the Viehweg-Zuo construction of Finsler pseudo-metrics in \cite{VZ-1}, two natural ingredients in the hyperbolic context: a pointwise argument plus a convex linear combination of components of the Viehweg-Zuo Finsler pseudo-metrics. {This allowed the Viehweg-Zuo construction to extend to the semi-ample case and and is adapted here for some necessary improvements of the Viehweg-Zuo construction in this paper.}}  \\[-3.5mm] 

{Another important motivation for our current investigation comes from the fact that many well-studied moduli spaces such as the moduli space of  abelian varieties suitably rigidified and that of marked K3 surfaces are actually locally symmetric of non-compact type and such locally symmetric varieties admit the famous Baily-Borel compactifications (cf.\cite{Baily-Borel}).  Borel showed in \cite{Borel-72} that, given such a locally symmetric variety, this compactification yields a hyperbolic embedding of the variety. In particular, the big Picard theorem holds on these moduli spaces.}

\begin{thm}[Borel]\label{borel}
Let $X$ be a torsion-free arithmetic quotient of a bounded symmetric domain. Denote by $X^*$ the Baily-Borel compactification of $X$. Then $X$ is hyperbolically embedded in $X^*$.  
\end{thm}

{For moduli spaces of canonically polarized varieties, the KSBA compactifications (see the survey \cite{Kollar-13}) are natural candidates for the hyperbolic embeddings and Borel's theorem leads us to pose:}
\begin{ques}
Let ${\bar f}:\,X \to Y$ be a KSBA stable family over a projective variety $Y$. Denote by $U\subset Y$ the open subset  over which $f$ is smooth. Is $U$ hyperbolically embedded in $Y$?  
\end{ques}

{Also inspired by this theorem, Javanpeykar and Kucharczyk in \cite{Java-Kuch} formulated the following:}
\begin{defi}
A finite type scheme $X$ over $\mathbb{C}$ is \emph{Borel hyperbolic} if, for every finite type reduced scheme $S$ over $\mathbb{C}$, any holomorphic map from $S$ to $X$ is algebraic.
\end{defi}

{It is easy to see that Borel hyperbolicity implies Brody hyperbolicity and that hyperbolic embeddability implies Borel hyperbolicity. As observed in \cite{Gri-King-73}, Borel hyperbolicity can be verified by restricting to the case $C$ is one dimensional. The following is a natural question in our context:\\[-4mm]
%
\begin{ques} Is every {Brody hyperbolic} moduli space of polarized manifolds Borel hyperbolic?
\end{ques}
}
{Very recently, Bakker, Brunebarbe and Tsimerman have obtained some sweeping partial results in this direction, see {\cite[Corollary~7.1]{BBT-18}}.} For a family with local Torelli injectivity, i.e. the period map is quasi-finite, the Borel hyperbolicity is a direct corollary of a conjecture of Griffiths. 
More precisely, let $B$ be {an algebraic variety} with a polarized variation of Hodge structures (PVHS) and $\Phi:\, B \to \Gamma \backslash D$ the induced period map. Here $D$ is the period domain (namely the classifying space of Hodge structures with fixed Hodge numbers) and $\Gamma$ is the monodromy group of the PVHS on $B$. In \cite{Gri-70}, Griffiths conjectured that the image $\Phi (B) \subset \Gamma \backslash D$ is quasi-projective. Note that the quotient space $\Gamma \backslash D$ is in general a highly transcendental object. {The paper \cite{BBT-18} confirms this conjecture assuming that $\Gamma$ is arithmetic as a corollary of its deep results on the o-minimal GAGA theorem.} 
We remark that this conjecture with arbitrary monodromy group $\Gamma$ was established {when} $\mathrm{dim}\,\Phi (B)=1$ (cf. \cite{Som-73}, \cite{CDK-95}) and $\mathrm{dim}\, B=2$ (cf. {\cite[Theorem~1.2.6]{GGRC-19}}). \\[-3mm]

Nevertheless, families of polarized varieties where the local Torelli injectivity fails abounds. 


\subsection{The big Picard theorem and Borel hyperbolicity}
Let ${\bar f}:\, X \to Y$ be an analytic family of projective manifolds over a projective base $Y$ with degeneration locus $S\subset Y$. {Hence ${\bar f}$ is a compactification of the smooth family $f:V\to U$ with $U = Y \setminus S$ and $V=\bar f^{-1}(U)$.} Griffiths introduced in \cite{Gri-1,Gri-2,Gri-3} the notion of polarized variation of Hodge structure on $U = Y \setminus S$. Schmid, Deligne and Cattani-Kaplan-Schmid  (\cite{Sch-73,Del-84,CKS-86}) have studied the asymptotic behavior of the Hodge structures and the Hodge metric near the degeneration locus. 
Their results are of fundamental importance in the study of the geometry of 
families. Kawamata and Viehweg's positivity theorems on the direct image  $f_* \,\omega^\nu_{X/Y} $ of powers of the relative dualizing sheaf are examples that play crucial roles in the investigation of the Iitaka conjecture. Another is Viehweg's work on constructing the moduli space of varieties with semi-ample dualizing sheaves.\\[-3mm]

As mentioned, the Torelli-type theorem fails in general for such a family.
As a substitute, Viehweg and the fourth named author constructed in \cite{VZ-1} a non-trivial comparison map between the usual Kodaira-Spencer map and the Kodaira-Spencer map on the Hodge bundles associated to some new family built from certain cyclic coverings of $X$. Consequently, using the semi-negativity of the kernels of the Kodaira-Spencer maps on the Hodge bundles (proven in \cite{Zuo-00}) and the positivity results on the direct image sheaves, the maximal non-zero iteration of Kodaira-Spencer map yields the ``\emph{bigness}" of the so-called \emph{Viehweg-Zuo subsheaves} in symmetric powers of $\Omega^1_Y(\log S)$.  
These subsheaves give rise analytically to negatively curved complex Finsler pseudometrics on $U=Y\setminus S$.\\[-3mm]

{In the paper \cite{Gri-King-73}, Griffiths and King studied the higher dimensional generalization of value distribution theory.  
{With it, they obtained a Nevanlinna-theoretic proof of Borel's theorem via negative curvature (cf.{\cite[Corollary~(9.22)]{Gri-King-73}}}). 
Our negatively curved Finsler pseudometrics 
led us naturally to this approach of Griffiths-King in generalizing Borel's theorem, which we manage to simplify and strenthen to yield the first proof of the key technical theorem of this paper below.} 
{\begin{thmx}[Criterion for the big Picard theorem]\label{thm:Big Picard}
	Let $X$ be a projective manifold, $\omega$ a K\"ahler metric on $X$  and $D$ a simple normal crossing divisor on $X$. Let $\gamma:{\mathbb D}^*\to X\setminus D$ be a  holomorphic map. 
Assume that there is a Finsler pseudometric $h$ on $T_X(-\log D)$ (in the sense of Definition \ref{def:Finsler}) such that $ |\gamma'(z)|_h^2\not\equiv 0$ 
and that the following inequality holds in the sense of currents
		\begin{align}\label{eq:condition}
		\hess \log |\gamma'(z)|_h^2\geq \gamma^*\omega.
		\end{align}
  Then $\gamma$ extends to a holomorphic map $\bar{\gamma}:{\mathbb D}\to X$.
\end{thmx}} 
We give a second perhaps more modern proof of this theorem in section 4, see also the remark just before that section. {We mention that this criterion 
is also used by the first named author in \cite{Den20}  to prove the big Picard theorem for varieties having a quasi-finite period map. {In addition, although the Nevanlinna theoretic tools involved are quite standard, our approaches in the two proofs of Theorem~\ref{thm:Big Picard} are not and have further implications beyond the criterion proper.} In particular, our proof can be modified to yield  Corollary~\ref{cor:extension} below, omitted for simplicity as it follows from a deep theorem of Siu (the Borel hyperbolicity part there being a simple corollary of Theorem~\ref{big-picard new}).\\[3mm] 
\indent
As indicated above, the construction of \cite{VZ-1} generalizes  from the case of canonically polarized manifolds to the  case of manifolds with semiample canonical bundle. This includes recent key observations from the paper of Popa-Taji-Wu \cite{PTW-18} and from the first named author in \cite{Deng-18b} and is worked out in section 2, where we construct the required metric $h$ over a desingularization of the base space of the family, i.e., one  that  satisfies (\ref{eq:condition}) of Theorem~\ref{thm:Big Picard}.  
This yields our main theorem.}\\[-4mm]
{\begin{thmx}[Big Picard theorem]\label{big-picard new}
		Let $(f:\,V \to U, \mathcal{L})\in \mathcal{M}_h(U)$ be an algebraic family of polarized projective manifolds of Hilbert polynomial  $h$ and semi-ample canonical bundle. Suppose that the moduli map $U\to M_h$ from $U$ to the coarse moduli space $M_h$ is quasi-finite. Given  a completion $\bar U$ of $U$,  any holomorphic map $\gamma:\, {\mathbb D}^* \to U$  extends to a holomorphic map $\bar \gamma:\, {\mathbb D} \to \bar U$. {In particular,  Any holomorphic map from an algebraic curve $T$ to $U$ is necessarily algebraic.}
\end{thmx}}
\vspace{-1mm}
%
\begin{corx}\label{cor:extension}
	Let $(f:\,V \to U, \mathcal{L})\in \mathcal{M}_h(U)$ be as given in Theorem \ref{big-picard new}.  Let $Y$ be a projective compactification of $U$. Then  any  holomorphic map $\gamma: \mathbb{D}^p\times (\mathbb{D}^*)^q\to U$ extends to a meromorphic map $\overline{\gamma}:\mathbb{D}^{p+q}\dashrightarrow Y$.  {In particular, $U$ is Borel hyperbolic: Any holomorphic map from an algebraic variety $T$ to $U$ is necessarily algebraic;} I.e., over an algebraic variety, any such holomorphic family of polarized projective manifolds is actually algebraic.
\end{corx} 
\vspace{-1mm}
{
\begin{proof}[Proof of part 1 of \em{Corollary~\ref{cor:extension}}] By \cite[Theorem 1]{Siu75}, any meromorphic map to a compact K\"ahler manifold extends across  a subvariety of   codimension 2. As $U$ embeds in ${\mathbb CP}^N,\exists\, N$, it suffices to prove the extension property for a holomorphic map of the form $\gamma: \mathbb{D}^{r}\times \mathbb{D}^*\to U$.  Now, the key result \cite[p.442,  ($\ast$)]{Siu75} of Siu states that $\gamma$ extends to a meromorphic map  $\overline{\gamma}:\mathbb{D}^{r+1}\dashrightarrow Y$ if for all $z$ in a subset of $\mathbb{D}^r$ of nonzero Lebesque measure, the holomorphic map  $\gamma|_{\{z\}\times \mathbb{D}^*}:\{z\}\times \mathbb{D}^*\to U$ can be extended to a holomorphic map from $\{z\}\times \mathbb{D}$ to $Y$. But this latter follows from Theorem~\ref{big-picard new}.
\end{proof}
\noindent
Note that, by  the first part of Theorem~\ref{big-picard new}, respectively that of Corollary~\ref{cor:extension}, the holomorphic map  $\gamma:T\to U$ extends to a meromorphic map between any of their projective compactifications, which is thus a rational map by Chow's theorem. The last parts of  Theorem~\ref{big-picard new} and Corollary~\ref{cor:extension} follow. \\[2mm]
	In fact Theorem~\ref{big-picard new} can be modified to a Lang conjecture type statement: for a family of polarized manifolds $f:\, V \to U$ with maximal variation in moduli, i.e. with generically finite moduli map, there is a proper subvariety of the base $U$ so that any punctured disk whose image is not contained in this proper subvariety satisfies the big Picard theorem. See Remark~\ref{relaxation} for details.}\\[0mm]
%
%

  \subsection{Algebraic hyperbolicity} \emph{Algebraic hyperbolicity} for a compact complex manifold $X$ was  introduced by Demailly in \cite[Definition 2.2]{Dem97}. {It is shown} in \cite[Theorem 2.1]{Dem97} that $X$ is algebraic hyperbolic if it is  Kobayashi hyperbolic.  The notion of algebraic hyperbolicity was generalized to the case of {smooth} log-pairs $(X,D)$ by Chen \cite{Che04}. {It naturally generalizes further to the case of arbitrary  singular pairs of (reduced) projective varieties:}
	\begin{defi}[Algebraic hyperbolicity]\label{def:DC}
		Let $X$ be a projective {variety} and $\Delta$ an algebraic subset.  For a reduced irreducible curve $C\subset X$ with $C\not\subset \Delta$ and  $\nu:\tilde{C}\to C$
		its normalization,  let $i_X(C,\Delta)$ be the number of 
		points in $\nu^{-1}(\Delta)$.  The pair $(X,\Delta)$ is \emph{algebraically hyperbolic}  if there is a K\"ahler metric $\omega$ on $X$ such that, for all curves $C\subset X$ as above,
\begin{align}\label{ineq}
		2g(\tilde{C})-2+i(C,\Delta)\geq {\rm deg}_{\omega}C:=\int_C\omega.
\end{align}	
	\end{defi} 
\vspace{1mm}
{It is easily seen, just as the fore-mentioned observation} by Demailly in the case $\Delta=\emptyset$, that if $X\setminus \Delta$ is hyperbolically embedded into $X$,
the pair $(X,\Delta)$ is  algebraically hyperbolic (c.f. \cite{PR07}). \\[-3mm]	
	
	Note that   $2g(\tilde{C})-2+i(C,\Delta)$    depends only on the complement $X\setminus\Delta$.
	Hence the above notion of hyperbolicity also makes sense for quasi-projective varieties: we say that a quasi-projective variety $U$ is algebraically hyperbolic if it is birational to $Y\setminus D$ for a smooth log pair $(Y,D)$   
{with a K\"ahler metric $\omega$ on $Y$ satisfying the inequality \ref{ineq} for all curves $C\subset Y$ such that $C\setminus D$ is finite over $U$.} \\[-3mm]

The last main theorem in this paper is the algebraic hyperbolicity of the moduli spaces considered.
	
	\begin{thmx}[Algebraic hyperbolicity]\label{thm:algebraic hyperbolicity}
		Let $(f:\,V \to U, \mathcal{L})\in \mathcal{M}_h(U)$ be a  polarized family as that given in Theorem \ref{big-picard new}. Then the base $U$  is  algebraically hyperbolic.\\[-3mm]
\end{thmx} 

\subsection{Outline}{
In Section~\ref{sec-VZ}, {we explain in some detail} the construction of Viehweg and the fourth named author in \cite{VZ-2,VZ-1} {with some necessary additions} in our more general setting. 
In sections~\ref{sect_BP} and \ref{sect_BP2}, we give two distinct Nevanlinna theoretic proofs of Theorems \ref{thm:Big Picard} and \ref{big-picard new}, the first being self-contained and  based on  the method of metric and curvature while the second is based on the  lemma on logarithmic derivatives. 
Section~\ref{alg_hyp} proves our theorem on algebraic hyperbolicity.} \\[-4mm]

\subsection{Notation} In general we follow the notations of \cite{VZ-2,VZ-1} on the construction of Higgs bundles. 
We write $u \gtrsim v$ if there exists a constant $c >0$ such that $u(s) \geq c \cdot v(s)$ for all $s \in S$.

\subsection{Acknowledgements}
 {The last three authors would like to thank Phillip Griffiths for his  interest and feedbacks, 
Ariyan Javanpeykar for his motivating original discussions,
Songyan Xie for helpful discussions, the Academy of Mathematics and Systems Science, the School of Mathematical Sciences at the East China Normal University, CIRGET and its members at the Universit\'{e} du Qu\'{e}bec \`{a} Montr\'{e}al for their hospitality.} The first named author would like to thank Professors Jean-Pierre Demailly, Min Ru, Nessim Sibony and Emmanuel Ullmo for their  discussions, encouragements,  and supports and the wonderful provisions of l'Institut des Hautes \'Etudes Scientifiques. 
\\[-3mm]

This article is the converged and expanded version of the two preprints \cite{LSZ19,Den19b} on the arXiv and hence takes precedence over them. 

\section{Recollections on the Viehweg-Zuo construction}\label{sec-VZ}
In \cite{VZ-1,VZ-2}, Viehweg and the fourth named author constructed two graded logarithmic Higgs bundles for any smooth family of projective manifolds $f_0:\, V \to U$ {in order to prove the hyperbolicity of the base $U$. We recall the construction with some simplifications and extensions.} 
\\[-3mm]

{We first note that $f_0$ is of maximal variation if $\mathrm{Var}(f_0) = \mathrm{dim}\, Y$, cf. {\cite[Section~1]{Kawa-85},} {where $\mathrm{Var}(f_0)$ is the dimension of the image of the moduli map. Therefore,}
if the base $U$ of the original family is not smooth, which is the case considered in Section~1 in general, we can always pull back the family by a desingularization of the base to one of smooth base with maximal variation. Since all the arguments from Viehweg-Zuo allow for such birational modifications, we assume from now on unless stated otherwise that the base $U$ of our family is non-singular, relaxing the quasi-finite hypothesis to that of maximal variation. We also assume, replacing $U$ when necessary by the Zariski closure of the holomorphic curve $C$, that the holomorphic curve $C$ under consideration is Zariski dense in $U$. Our proof of Theorem~\ref{big-picard new} then proceeds without loss of generality (see Remark~\ref{rem:proof}).} 

\subsection{Cyclic covering and the comparison map}\label{comparison1}
Let $f:V \to U$ be a smooth algebraic family of polarized projective manifolds with semi-ample canonical bundles and with $U$ nonsingular. 
Then there is a partial good compactification ${\bar f}: X \to Y$ of the this family, meaning {by definition} that: 
\begin{itemize}
\item[1)] $X$ and $Y$ are quasi-projective manifolds, and $U \subset Y$. 
\item[2)] $S:=Y \setminus U$ and $\Delta:={({\bar f}^*S)_{red}}$ are normal crossing divisors.
\item[3)] ${\bar f}$ is a log smooth projective morphism between the log pairs $(X,\Delta)$ and $(Y,S)$, and ${\bar f}^{-1}(U) \to U$ coincides with the original family $V \to U$.
\item[4)] $Y$ has a non-singular projective compactification $\bar{Y}$ such that $\bar{Y} \setminus U$ is a normal crossing divisor and $\mathrm{codim}(\bar{Y} \setminus Y) \geq 2$.
\end{itemize}
{We fix such a compactification and adopt the notations as given in this definition.} \\[-3mm]

The strategy of Viehweg and Zuo \cite{VZ-1,VZ-2} is to exploit the Kawamata-Viehweg positivity result of the direct image sheaves, which they used to first construct cyclic coverings over the total space $X$ that are natural with respect to the family. {This construction is accomplished through:}

\begin{thm}[{\cite[Corollary~4.3]{VZ-1}} or {\cite[Proposition~3.9]{VZ-2}}]
Denote by $\mathcal{L}:= \Omega^n_{X/Y}(\log \Delta)$ the sheaf of top relative log differential forms $(n$ is the relative dimension$\,)$. Suppose that $f={\bar f}|_U$ is of maximal variation. Then there exists an ample line bundle $A$ on $\bar{Y}$ and an integer $\nu \gg 1$ such that $\mathcal{L}^{\nu} \otimes f^*A^{-\nu}$ is globally generated over $V_0:=f^{-1}(U_0)$, where $U_0$ is some open dense subset of $U$. 
\end{thm}
It follows that the invertible sheaf $\mathcal{L}^{\nu} \otimes f^*A^{-\nu}$ has plenty of nontrivial sections {for $\nu$ large}. A cyclic covering of $X$ is thus obtained by taking the $\nu$-th roots of such a section $s$. We choose $Z$ to be a desingularization of this cyclic covering and denote the induced morphisms by $\psi:\, Z \to X$ and $g:\, Z \to Y$. {The  new family $g$ has in general a larger discriminant locus than $S$ given by $S \cup T$ where $T$ is the discriminant of  the zero divisor $H=(s)$ over $Y$. Thus} the restriction of $g$ over $Y \setminus (S \cup T)$ is smooth, which we denote by $g_0:\, Z_0 \to U_0$.

\subsubsection{The Higgs bundle coming from {the} variation of Hodge structures}
Consider the VHS on $U_0$ induced by the local system ${ \mathcal{V}_0=}R^n{g_0}_*\mathbb{C}_{Z_0}$. By blowing up the closure $\bar{S}+\bar{T}$ of $S+T$ in $\bar{Y}$ if necessary and replacing $\bar{S}+\bar{T}$ by its preimage, we assume that $\bar{S}+\bar{T}$ is a simple normal crossing divisor.  Deligne's quasi-canonical extension {(cf. \cite[\S4]{Sch-73})} then applies and we get a locally free sheaf $\mathcal{V}$ on $\bar{Y}$ {extending $ \mathcal{V}_0$} with the Gauss-Manin connection:
\[
\nabla:\, \mathcal{V} \to \mathcal{V} \otimes \Omega^1_{\bar{Y}}(\log(\bar{S}+\bar{T})).
\]
By the nilpotent orbit theorem (in \cite{Sch-73} and \cite{CKS-86}), the Hodge filtration {$\{ \mathcal{F}_0^p\}$ of $\mathcal{V}_0$ extends as a filtration by subbundles $\{ \mathcal{F}^p\}$ of $\mathcal{V}$} so that the associated Hodge bundle 
$
E := \mathrm{Gr}_{\mathcal{F}^{\bullet}}\mathcal{V}
$ 
is locally free on $\bar{Y}$. The induced Higgs map
\[
\theta:= \mathrm{Gr}_{\mathcal{F}^{\bullet}}\nabla: \, E \to E \otimes \Omega^1_{\bar{Y}}(\log(\bar{S}+\bar{T}))
\]
has logarithmic poles along $\bar{S}+\bar{T}$. One can write the Hodge bundle {summands}  as higher direct image sheaves of log forms if the divisor $\bar{S}+\bar{T}$ is smooth (cf. \cite{Zucker}). 
More precisely, we have
\[
E^{p,q}|_{Y_0} \cong R^qg_* \Omega^p_{Z/Y}(\log \Pi)|_{Y_0} 
\]
for $Y_0:= Y \setminus \mathrm{Sing}(\bar{S}+\bar{T})$ and($q:=n-p$), where $\Pi:=g^{-1}(S \cup T)$ ($\Pi$ is assumed to be normal crossing after birational modification of $Z$). {Clearly,} $\mathrm{codim}(\bar{Y} \setminus Y_0) \geq 2$.
\begin{rmk}
In the construction above we have already changed the birational model of $U$ since we have to blow up $T$ inside $U$. As mentioned,  this is allowed in our application.  
\end{rmk}

\subsubsection{The Higgs bundle coming from deformation theory}
The Hodge bundle $(E,\theta)$ {so constructed above has in general extra, artificially introduced logarithmic poles along $T\cap U$.} To study the hyperbolicity of the original base space $U$, we shall construct a Higgs bundle directly from the original family of maximal variation, whose Higgs map has logarithmic poles only along the boundary $S$. {Recall that we have assumed that $U$ is smooth.}\\[2mm]
As in \cite{VZ-1} and \cite{VZ-2}, we shall use the tautological short exact sequences
\begin{equation}\label{tautol-seq}
0 \to f^*\Omega^1_Y(\log S) \otimes \Omega^{p-1}_{X/Y}(\log \Delta) \to \mathfrak{gr}(\Omega^p_X(\log \Delta)) \to \Omega^p_{X/Y}(\log \Delta) \to 0
\end{equation}
where
\[
\mathfrak{gr}(\Omega^p_X(\log \Delta)) := \Omega^p_X(\log \Delta) / f^*\Omega^2_Y(\log S) \otimes \Omega^{p-2}_{X/Y}(\log \Delta).
\]
Note that the short exact sequence can be established only when $f:\,(X,\Delta) \to (Y,S)$ is log smooth. Denote by $\mathcal{L}=\Omega^n_{X/Y}(\log \Delta)$ as before. We define {on $Y$ the reflexive Higgs sheaf}
\[
F^{p,q}_0:= R^qf_*(\Omega^p_{X/Y}(\log \Delta) \otimes \mathcal{L}^{-1})/\mathrm{torsion}
\]
together with the edge morphisms
\[
\tau^{p,q}_0:\, F^{p,q}_0 \to F^{p-1,q+1}_0 \otimes \Omega^1_Y(\log S)
\]
induced by the exact sequence (\ref{tautol-seq}) tensored with $\mathcal{L}^{-1}$.  
\begin{rmk}
It is easy to see that $\tau^{n,0}_0|_U$ is nothing but the Kodaira-Spencer map of the family. So the Higgs maps $\tau^{p,q}_0$ can be regarded as the \emph{generalized Kodaira-Spencer maps}.  
\end{rmk}
We denote by $F^{p,q}$ the reflexive hull of $F^{p,q}_0$ on $\bar{Y}$. The Higgs maps $\tau^{p,q}_0$ extends automatically since $\mathrm{codim}(\bar{Y}\setminus Y) \geq 2$. So we get the {extended reflexive Higgs sheaf $(F,\tau)$,} defined on $\bar{Y}$.

\subsubsection{The comparison maps}
In \cite{VZ-1,VZ-2} Viehweg and Zuo constructed the following comparison maps $\rho^{p,q}$, which connect $(F,\tau)$ and $(E,\theta)$.
\begin{lem}
Using the same notations {as} introduced above, let
\[
\iota :\, \Omega^1_{\bar{Y}}(\log \bar{S}) \to \Omega^1_{\bar{Y}}(\log (\bar{S} + \bar{T}))
\]
be the natural inclusion. Then there exists morphisms $\rho^{p,q}:\, F^{p,q} \to A^{-1} \otimes E^{p,q}$ such that the following diagram commutes.
\begin{equation}
  \label{eq:3}
\xymatrixcolsep{5pc}\xymatrix{
F^{p,q} \ar[r]^-{\tau^{p,q}} \ar[d]^{\rho^{p,q}} & F^{p-1,q+1} \otimes \Omega^1_{\bar{Y}}(\log \bar{S}) \ar[d]^{\rho^{p-1,q+1} \otimes \iota} \\
A^{-1} \otimes E^{p,q} \ar[r]^-{\mathrm{id} \otimes \theta^{p,q}} &  A^{-1} \otimes E^{p-1,q+1} \otimes \Omega^1_{\bar{Y}}(\log (\bar{S} + \bar{T}))
}  
\end{equation}  
\end{lem}
\begin{rmk}{
\emph{A priori}, our comparison map $\rho^{p,q}$ is defined only on $Y_0$. That is, it is a morphism between $F^{p,q}|_{Y_0} \cong R^qf_*(\Omega^p_{X/Y}(\log \Delta) \otimes \mathcal{L}^{-1})|_{Y_0}/\mathrm{torsion}$ and $E^{p,q}|_{Y_0} \cong R^qg_* \Omega^p_{Z/Y}(\log \Pi)|_{Y_0}$. But as $F^{p,q}$ is reflexive, $E^{p,q}$ locally free and $\mathrm{codim}(\bar{Y} \setminus Y_0) \geq 2$, each comparison map $\rho^{p,q}$ extends to $\bar Y$.}
\end{rmk}

\subsubsection{Injectivity of the comparison map}\label{inj-comparison}
In order to use the comparison map to construct a negatively curved Finsler pseudometric, one {needs some} pointwise injectivity   of the 
comparison map
$$(\rho^{n-1,1}\otimes \iota)\circ \tau^{n,0}: \mathcal{O}_{\bar{Y}}=F^{n,0}\to A^{-1}\otimes E^{n-1,1}\otimes \Omega^1_{\bar{Y}}(\log  \bar{S})$$ 
{on $U\setminus T$ as maps between vector bundles,} as well as that of the induced map
\begin{align}\label{eq:induced}
\tau^1:T_{\bar{Y}}(-\log \bar{S})\to A^{-1}\otimes E^{n-1,1}.
\end{align}

Denote by $\rho^{p,q}_y$ the restriction of {$\rho^{p,q}|_{U \setminus T}$, as a map between vector bundles,} at a point $y\in {U \setminus T}$. In \cite{VZ-1}, the following fact on the injectivity of $\rho^{p,q}_y$ are obtained:
\begin{itemize}
\item[1)] $\rho^{n,0}_y$ is injective for every $y \in U \setminus T$.
\item[2)] If the family is canonically polarized, then $\rho^{p,q}_y$ is an injection for each $(p,q)$ with $p+q=n$ and for every $y\in U\setminus T$.
\end{itemize}

We {remark} that the injectivity of $\rho^{p,q}$ for all $p, q$ follows from the Kodaira-Akizuki-Nakano vanishing theorem, but that the injectivity of $\rho^{n-1,1} $ (as first observed in  \cite{PTW-18}) {only needs the Bogomolov-Sommese vanishing theorem. As the latter vanishing theorem holds for} projective maniofolds of general type, and as the canonical bundle of such a manifold is semiample if and only if it is nef, one obtains:
\begin{thm}
Let $(f:\,V \to U, \mathcal{L})\in \mathcal{M}_h(U)$ be a maximally varying family of polarized projective manifolds with nef and big canonical bundle. Then $(\rho^{n-1,1}\otimes \iota)\circ \tau^{n,0}|_{U \setminus T}$, as a map of vector bundles, is injective at every point in $U \setminus T$ where the Kodaira-Spencer map is injective.
\end{thm}
In our {general} setting of families of polarized manifold with semi-ample canonical divisors, we have the following key theorem of Viehweg-Zuo.

\begin{thm}[Viehweg-Zuo]\label{alg-inj}
Let $(f:\,V \to U, \mathcal{L})\in \mathcal{M}_h(U)$ be the polarized family as that in Theorem \ref{big-picard new}. Then the map $(\rho^{n-1,1}\otimes \iota)\circ \tau^{n,0}$ along any algebraic curve $\gamma: C\to U$  does not vanish.  
\end{thm}
We briefly outline the proof of this theorem, 
which used a global argument relying on the Griffiths curvature computation for Hodge metrics: Suppose that $(\rho^{n-1,1}\otimes \iota)\circ \tau^{n,0}$ vanishes along $\gamma(C)$. Then the image $\mathcal O_Y=F^{n,0}\subset E^{n,0}\otimes A^{-1}$ lies in $\mathrm{Ker}(\theta^{n-1,1}) \otimes A^{-1}$.
Note that $\mathrm{Ker}(\theta^{n-1,1})$ is semi-negatively curved for the degenerated Hodge metric (cf. \cite{Zuo-00}), which essentially follows from the Griffiths curvature computation. By taking integration of the curvature form of Hodge metric restricted to $\mathcal O_Y$ one shows that the trivial line bundle is strictly negative because of the curvature decreasing property 
 for holomorphic subbundles. This is of course a contradiction.\\[-3mm]

Very recently, {the first named author observed that the argument of Viehweg-Zuo can be made pointwise, combining with a usual maximal principle argument. His argument runs as follows:
instead of taking integration of the curvature form of Hodge metric, he evaluated} the curvature form on a special point $y_0$ in $U$, where the norm function of the constant section of $\mathcal O_Y$ with respect to the Hodge metric of $E^{n,0}\otimes A^{-1}$ restricted to $\mathcal O_Y$ via $\rho^{n,0}:\, \mathcal{O}_Y \hookrightarrow E^{n,0} \otimes A^{-1}$ { takes the maximal value, the metric being in fact a natural modification by Popa-Taji-Wu in \cite{PTW-18} of the Viehweg-Zuo's metric on $U$ to remove its singularities on $T$. This implies that the curvature form on $\mathcal O_Y$ is semi-positive at this point.}
 On the other hand, if the map $(\rho^{n-1,1}\otimes \iota)\circ \tau^{n,0}$ at this specific point  evaluated in some tangent vector vanishes, then the Griffiths curvature formula and the strict negativity of $A^{-1}$ implies that the the curvature form  on $\mathcal O_Y$ at $y_0$ is strictly negative along this tangent vector, which gives us a contradiction. {We thus have:}

\begin{thm}[\cite{Deng-18}]\label{deng}
Let $f:V\to U$ {be a family of {polarized projective} manifolds with semi-ample canonical sheaf over a {nonsingular} quasiprojective variety $U$. Assume that $f$ is of maximal variation.}  Then the map $\tau^1$ defined in \eqref{eq:induced} is a vector bundle injection on a Zariski open subset $U^\circ$ of $U$. In particular, the analytic version of Theorem~\ref{alg-inj} holds true{; I.e.,  }
the map $(\rho^{n-1,1}\otimes \iota)\circ \tau^{n,0}$ {does not vanish identically} along any holomorphic curve $\gamma: \, C \to U$ with $\gamma(C)\cap U^\circ\neq \varnothing$.
\end{thm}

\begin{conj}\label{pt-inj}
{With the assumption as that in Theorem \ref{deng},}  the map $(\rho^{n-1,1}\otimes \iota) \circ \tau^{n,0}$ is a vector bundle injection {over the open subset of $U$ with quasi-finite moduli map}.
\end{conj}
Conjecture~\ref{pt-inj} has been verified for such a family with members of Kodaira dimension one in a joint paper of the third and fourth named authors with Xin Lu \cite{LSZ}.

\subsection{Iteration of generalized Kodaira-Spencer maps}
We iterate the Higgs maps to get
\[
\tau^{n-q+1,q-1} \circ \cdots \circ \tau^{n,0} :\, F^{n,0} \to F^{n-q,q} \otimes \bigotimes^q \Omega^1_{\bar{Y}}(\log \bar{S}).
\]
{Since the Higgs field $\tau$ satisfies $\tau \wedge \tau=0$, this {iterated} composition  factors through the map}
\[
\tau^q:\, F^{n,0} \to F^{n-q,q} \otimes Sym^q\, \Omega^1_{\bar{Y}}(\log \bar{S}).
\]
As $\mathcal{O}_{\bar{Y}}$ is a subsheaf of $F^{n,0}$, the {following} composition of maps {makes sense}
\[
\xymatrix{
Sym^q \, T_{\bar{Y}}(-\log \bar{S}) \ar[r]^-{\subset} & F^{n,0} \otimes Sym^q \, T_{\bar{Y}}(-\log \bar{S}) \ar[r]^-{\tau^q \otimes \mathrm{id}}  &  F^{n-q,q} \otimes Sym^q\, \Omega^1_{\bar{Y}}(\log \bar{S}) \otimes Sym^q \, T_{\bar{Y}}(-\log \bar{S}) \ar[d]^{\mathrm{id} \otimes <,>} \\
  &  &  F^{n-q,q}.
}
\]
{We will still denote this map as $\tau^q$ by abuse of notation.}\\[2mm]
Composing this $\tau^q$ with the comparison map $\rho^{n-q,q}$, we get the \emph{iterated Kodaira-Spencer map}
\begin{equation}
  \label{eq:4}
Sym^q\, T_{\bar{Y}}(-\log \bar{S}) \xrightarrow{\tau^q} F^{n-q,q} \xrightarrow{\rho^{n-q,q}} A^{-1} \otimes E^{n-q,q}. 
\end{equation}

\subsubsection{Maximal non-zero iteration}
We define the \emph{maximal non-zero iteration of Kodaira-Spencer map} to be the
the $m$-th iterated Kodaira-Spencer map with $\rho^{n-m,m} \circ \tau^m(Sym^m \, T_{\bar{Y}}(-\log \bar{S})) \neq 0$ {and with $m$ being the largest number $m_f$ satisfying this property, if it exists.} More precisely,
\[
\xymatrix{
Sym^q\, T_{\bar{Y}}(-\log \bar{S}) \ar[r]^-{\rho^{n-m,m}\circ \tau^m \neq 0}  \ar@/_2pc/[rd]^-{=0} & A^{-1} \otimes E^{n-m,m} \ar[d]^{\mathrm{id} \otimes \theta^{n-m,m}} \\
  &  A^{-1} \otimes E^{n-m-1,m+1} \otimes \Omega^1_{\bar{Y}}(\log (\bar{S} + \bar{T})).
}
\]
Note that $\{0\}\ne \mathrm{Im}(\rho^{n-m,m} \circ \tau^m) \subset A^{-1} \otimes \mathrm{Ker}(\theta^{n-m,m})$. We call this number ${m_f}$ the \emph{maximal length of iteration} or the \emph{maximal iteration length}.
\begin{lem}\label{gen-inj}
Keeping the   assumptions as {\ref{deng}}, we have that $\rho^{n-1,1} \circ \tau^1$ is injective at the generic point of $U$, evaluated at each tangent vector. 
\end{lem}
\begin{proof}
This follows from Theorem~\ref{alg-inj} when the canonical divisor of a general fiber of the family is semi-ample and big and from Theorem~\ref{deng} {in general}.
\end{proof}
\begin{cor}
{Non-zero such iteration on $Y$ exists with maximal iteration length  $m_f \leq n$.}
\end{cor}
\begin{proof}
 $\rho^{n-1,1} \circ \tau^1$ is non-zero by Lemma~\ref{gen-inj}. The upper bound of $m$ follows from $\theta^{0,n}=0$.  
\end{proof}

\subsubsection{Maximal non-zero iteration of Kodaira-Spencer map along an analytic curve}\label{max-iteration-curve}
Let $V\to U$ be the smooth family  as that in Theorem \ref{deng}.  In our application we {only need to} consider an analytic map $\gamma$ from a complex analytic curve $C$ to the base manifold $U$ so that $\gamma(C)\cap U^\circ\neq\varnothing$, where $U^\circ$ is the dense Zariski open set of $U$ in Theorem \ref{deng}.  All the Higgs bundles can be pulled back to $C$, as well as the iteration process. We define the composition
\[
\tau^{p,q}_{\gamma}:\, \gamma^*F^{p,q} \xrightarrow{\gamma^*\tau^{p,q}} \gamma^*F^{p-1,q+1} \otimes \gamma^* \Omega^1_{\bar{Y}}(\log \bar{S}) \xrightarrow{\mathrm{id} \otimes \mathrm{d}\gamma} \gamma^*F^{p-1,q+1} \otimes \Omega^1_C
\]
as the \emph{Higgs map along $\gamma$}. We define $\theta_{\gamma}$ similarly. Then $(\gamma^*F,\tau_{\gamma})$ and $(\gamma^*E,\theta_{\gamma})$ are holomorphic Higgs bundles on the curve $C$. The iterated Kodaira-Spencer maps on $C$ are also defined similarly:
\[
T^{\otimes q}_C \xrightarrow{\tau^q_{\gamma}} \gamma^*F^{n-q,q} \xrightarrow{\gamma^*\rho^{n-q,q}} \gamma^*A^{-1} \otimes \gamma^*E^{n-q,q}.
\]
\begin{cor}  {Non-zero iteration of such maps along $\gamma$ exists. {I.e., it has length} $m' \geq 1$.}
\end{cor}

\begin{proof}
Since $\gamma(C)\subset U$ is Zariski dense, we know from Lemma~\ref{gen-inj} that
$\rho^{n-1,1} \circ \tau^1$ is injective at all the points of $\gamma(C)$  contained in a Zariski open subset of $U$, evaluated at all {the} tangent directions at those points. This implies that at least $\gamma^*(\rho^{n-1,1} \circ \tau^1_{\gamma})(T_C)$ is non-zero.
\end{proof}

By its definition, the maximal non-zero iteration of Kodaira-Spencer map along  $\gamma$ of length $m'$ has the properties: $\gamma^*(\rho^{n-m',m'} \circ \tau^{m'}_{\gamma})(T^{\otimes m'}_C) \neq 0$  and  $\mathrm{Im}(\gamma^*(\rho^{n-m',m'} \circ \tau^{m'}_{\gamma})) \subset \gamma^*A^{-1} \otimes \mathrm{Ker}(\theta_{\gamma})$.  {These} properties are crucial for creating a negatively curved Finsler pseudometric along $\gamma$ in section~\ref{sec_finsler-metric}.
\begin{rmk}
One should be warned that the maximal length of iteration along $\gamma$ could be \emph{shorter} than the maximal length of iteration of the original family. This is because the iterated Kodaira-Spencer maps $Sym^q\, T_{\bar{Y}}(-\log \bar{S}) \xrightarrow{\rho^{n-q,q} \circ \tau^q} A^{-1} \otimes E^{n-q,q}$ with $q>1$ are not injective in general.\\[1mm]
%
\end{rmk}

\subsection{The Finsler (pseudo)metric}\label{sec_finsler-metric}
\begin{defi}[Finsler metric]\label{def:Finsler}
	Let \(E\) be a  holomorphic vector bundle on a complex manifold $X$. A \emph{Finsler pseudometric} on \(E\) is a 
	 \emph{continuous}  function \(h:E\to  \mathclose[ 0,+\infty\mathopen[ \) such that 
	\[h(av) = |a|h(v)\]
	for any \(a\in \C \) and \(v\in  E\). We call $h$ a Finsler metric if it is nondegenerate, i.e, if $h(v)=0$ only when $v=0$.
\end{defi} 
\noindent
\ \ As in \cite{VZ-1}, we construct a Finsler pseudometric via the maximal iterated Kodaira-Spencer map
\begin{equation}
  \label{eq:5}
Sym^m \, T_{\bar{Y}}(-\log \bar{S}) \xrightarrow{\rho^{n-m,m} \circ \tau^m} A^{-1} \otimes E^{n-m,m}.  
\end{equation}
It is given as follows. Consider $g_{A^{-1}} \otimes g_{hod}$ on $A^{-1} \otimes E^{n-m,m}$, where $g_A$ is the Fubini-Study metric of the ample line bundle $A$ and $g_{hod}$ is the Hodge metric on the Hodge bundle $E$. Pulling it back via $\rho^{n-m,m} \circ \tau^m$ and taking {the} $m$-th root, we get our desired Finsler pseudometric on $T_{\bar{Y}}(-\log \bar{S})$.

\subsubsection{Modification along the boundary}\label{rmk*}
In fact, Viehweg and the fourth named author used some modified version of the Finsler pseudometric described above in order to have the right kind of curvature property. 
This method of {modifying metrics on the log-tangent bundle} appeared first in the second named author's thesis {on} extending meromorphic maps \cite{Lu-91}.\\[2mm]
First we construct an auxiliary function associated to the boundary divisor $\bar{S}$. Denote by $\bar{S}_1,\dots, \bar{S}_p$ the  
components of $\bar{S}$. Let $L_i$ be the line bundle with section $s_i$ such that $\bar{S}_i=\mathrm{div}(s_i)$. Equip each $L_i$ with a smooth hermitian metric $g_i$. Let $l_i:= -\text{log}\,\lvert s_i \lvert^2_{g_i}$ and $l_S:= l_1l_2\cdots l_p$. Recall that the Hodge metric $g_{hod}$ has extra degeneration along $\bar{T}$ since the Hodge bundle $(E,\theta)$ has logarithmic poles along $\bar{T}$. To control the asymptotic behaviour of $g_{hod}$ near $\bar{T}$, we construct another auxiliary function associated to the divisor $\bar{T}$, in a similar manner as for $l_S$: Denote by $\bar{T}_1,\dots, \bar{T}_q$ the non-singular components of $\bar{T}$. Let $L'_i$ be the line bundle with section $t_i$ such that $\bar{T}_i=\mathrm{div}(t_i)$. Equip each $L'_i$ with a smooth hermitian metric $g'_i$. Let $l'_i:= -\text{log}\,\lvert t_i \lvert^2_{g'_i}$ and $l_T:= l'_1l'_2\cdots l'_q$.\\[1mm]
Now for each positive integer $\alpha$, we define a new singular hermitian metric $g_{\alpha}:= g_A \cdot l^{\alpha}_S \cdot l^{\alpha}_T$ on $A$.\\[-3mm]

Before entering our setting of the Viehweg-Zuo construction, we 
remark that by suitably modifying a (pseudo)metric on the tautological line bundle of the projective log tangent bundle satisfying the hypothesis of Theorem~\ref{thm:Big Picard}, one can already obtain the strongly negative curvature property 
that the holomorphic sectional curvature is bounded from above by a negative constant.
\begin{prop}\label{abstract-curv-bound}
Let $(X,D)$, $\gamma:\, \mathbb{D}^* \to X \setminus D$ be the same as in Theorem~\ref{thm:Big Picard}. Let $h$ be the Finsler pseudometric on $T_X(-\log D)$ (or equivalently, a semi-norm on $\mathcal{O}_{{\mathbb P}_{\text{alg}}(T_X(-\log D)^\vee)}(-1)$) satisfying the curvature inequality $(\ref{eq:condition})$. Then there is a Finsler pseudometric $h_{\alpha}$ defined as $h \cdot l^{-\alpha}_D$, where $l_D$ is the above auxiliary function of the boundary divisor $D$ and $\alpha$ is some positive integer, such that its holomorphic sectional curvature is bounded from above by a negative constant, i.e.:
\[
-dd^c\log \lvert \gamma'(z) \lvert^2_{h_{\alpha}} \lesssim - \mu_\alpha
\] 
where $\mu_\alpha$ is the semi-positive $(1,1)$ form associated with the possibly degenerate hermitian metric $(\mathbb{P}\gamma')^{-1} (\lvert\ \lvert_{h_{\alpha}}\circ \gamma_*)^2=((\mathbb{P}\gamma')^{-1}{h_{\alpha}})|\gamma_*|^2$ on $\mathbb{D}^*$ and the inequality above holds in the sense of currents.
\end{prop}

\begin{proof}
By direct computation we have 
\begin{align}
\begin{array}{l}
dd^c\log \lvert f'(z) \lvert^2_{h_{\alpha}} = dd^c\log \lvert f'(z) \lvert^2_h  
 - \alpha f^* \Sigma \frac{\mathrm{d} \mathrm{d}^cl_i}{l_i}  + \frac{\sqrt{-1}}{2\pi} \alpha f^* \Sigma \frac{\partial l_i \wedge \overline{\partial l_i}}{l^2_i} \\
\geq dd^c\log \lvert f'(z) \lvert^2_h  
 - \alpha f^* \Sigma \frac{\mathrm{d} \mathrm{d}^cl_i}{l_i} 
\geq f^* \left( \omega_{FS} - \alpha \Sigma \frac{\mathrm{d} \mathrm{d}^cl_i}{l_i}  \right).
\end{array}
\end{align}
Here we have used the inequality $dd^c\log \lvert f'(z) \lvert^2_h \geq f^*\omega_{FS}$. And by the same argument in {\cite[\S4, Proposition~1]{Lu-91}} (see also the proof of Lemma~7.1 in \cite{VZ-1}), we can find a positive definite Hermitian
form $\omega_{\alpha}$ on $T_X(- \log D)$ such that
\[
\omega_{FS} - \alpha \Sigma \frac{\mathrm{d} \mathrm{d}^cl_i}{l_i}  \geq l^{-2}_D  \cdot \omega_{\alpha}.
\]  
Note that $l^{-2}_D \cdot \omega_{\alpha}$ can also be regarded as a semi-norm on the dual of the tautological line bundle. So if we choose $\alpha >2$, then $l^{-2}_D  \cdot \omega_{\alpha} \gtrsim h_{\alpha}$ by the compactness of $X$. Therefore, we get the desired bound on the negative holomorphic sectional curvature 
$
-dd^c\log \lvert f'(z) \lvert^2_{h_{\alpha}} \lesssim - (\mathbb{P}f')^{-1} h_{\alpha}.
$ 
\end{proof}

Now we come back to our setting of the Viehweg-Zuo construction. We first note that

\begin{lem}\label{curv-estimate}
$\Theta(A,g_{\alpha})$ dominates 
the K\"ahler form $\omega_{FS}: = \Theta(A,g_A)$ on $\bar{Y}$ as currents.
\end{lem}
\begin{proof}
This follows from the computation
\begin{align*} 
\Theta(A,g_{\alpha}) &= \Theta(A,g_A \cdot l^{\alpha}_S \cdot l^{\alpha}_T)\\ &= \Theta(A,g_A) - \alpha \Sigma \frac{\mathrm{d} \mathrm{d}^cl_i}{l_i} - \alpha \Sigma \frac{\mathrm{d} \mathrm{d}^cl'_i}{l'_i} + \frac{\sqrt{-1}}{2\pi} \alpha \Sigma \frac{\partial l_i \wedge \overline{\partial l_i}}{l^2_i} + \frac{\sqrt{-1}}{2\pi} \alpha \Sigma \frac{\partial l'_i \wedge \overline{\partial l'_i}}{{l'}^2_i}\\
&\geq \Theta(A,g_A) - \alpha \Sigma \frac{\mathrm{d} \mathrm{d}^cl_i}{l_i} - \alpha \Sigma \frac{\mathrm{d} \mathrm{d}^cl'_i}{l'_i} \geq c \cdot \Theta(A,g_A), 
\end{align*}
which holds for some positive constant $c$. Note that one can rescale $g_i$ (respectively $g'_i$) to make $l_i$ (respectively $l'_i$) sufficiently large and leave $\mathrm{d} \mathrm{d}^cl_i$ (respectively $\mathrm{d} \mathrm{d}^cl'_i$) unchanged.
\end{proof}

\begin{rmk} (1) The same computations as above show, by the hypothesis on $h$, that $dd^c\log \lvert f'(z) \lvert^2_{h_{\alpha}}$ dominates $dd^c\log \lvert f'(z) \lvert^2_h$ and $f^*\omega_{FS}$ as currents. 
(2) As we mentioned above, the second named author used this type of modification of metrics to prove his extension theorem (cf. {\cite[\S4]{Lu-91}}). 
 Later, Viehweg and the fourth named author applied it  in {\cite[\S7]{VZ-1}} to the Viehweg-Zuo  metric of the family over $U$ and obtained the curvature estimate in Lemma~\ref{curv-estimate}. Since the family concerned in \cite{VZ-1} is canonically polarized, one can move the branch divisor of the cyclic covering such that the discriminant locus $T$ intersects with the analytic curve $\gamma(C)$ only at the smooth part of $T$, and the intersection is transversal. Then, the monodromy of the pull-back local system around $\gamma^*T$ is finite, and the pull-back Hodge metric $\gamma^*g_{hod}$ is bounded (see section~5 of \cite{VZ-1} for details). Popa-Taji-Wu also observed that a similar modification along $\bar{T}$ applies without violating the curvature estimate and chose $\alpha$ sufficiently large such that the singular hermitian metric $g^{-1}_{\alpha} \otimes g_{hod}$ is bounded (in fact continuous) near $\bar{T}$ without the canonically polarized hypothesis (cf. {\cite[\S3.1]{PTW-18}}). This means that the Finsler pseudometric on the log tangent bundle induced by $g^{-1}_{\alpha} \otimes g_{hod}$ is bounded, which is important for our curvature estimate. Note that here we use the property that the Hodge metrics have at most logarithmic growth along $\bar{S} + \bar{T}$, which is guaranteed by the study of the higher dimensional asymptotic behavior of the Hodge metric in \cite{CKS-86}.    
\end{rmk}

\subsubsection{The curvature inequalities}\label{sec-curv-bound}
Now we consider an analytic map $\gamma$ from a Riemann surface $C$ (for instance, $C= \mathbb{D}^*$) to the base manifold $U$. Let $m$ be the maximal length of iteration along $\gamma$.\\[-3mm]

It is very natural to use the hermitian metric $g^{-1}_{\alpha}\otimes g_{hod}$ and the iterated Kodaira-Spencer map to construct a Finsler pseudometric $F_{\alpha}$ on $T_{\bar{Y}}(-\log \bar{S})$:
\begin{equation}\label{finsler-metric}
\lvert v \lvert^2_{F_{\alpha}} := \lvert \rho^{n-m,m}\circ \tau^m(v^{\otimes m}) \lvert^{2/m}_{g^{-1}_{\alpha} \otimes g_{hod}},  \textrm{      for $v \in T_{\bar{Y}}(-\log \bar{S})$.}
\end{equation}

We first state a curvature inequality associated to $F_{\alpha}$ \emph{which validates the hypothesis in our criterion for big Picard theorem}. This inequality is given, though not explicitly, in {\cite[\S7]{VZ-1}} and holds in general once we have Theorem \ref{deng} in hand. We repeat the proof here only for consistency.
\begin{thm}[The curvature inequality]\label{curv-ineq-for-criteria}
Let  $V \to U$ be the same family as that in Theorem~\ref{deng}.
Fix the analytic map $\gamma:\, C \to U$ so that $\gamma(C)\cap U^\circ\neq\varnothing$ where $U^\circ$ is the dense Zariski open set of $U$ in Theorem \ref{deng} so that $\tau^1|_{U^\circ}:T_{U^\circ}\to A^{-1}\otimes E^{n-1,1}|_{U^\circ}$ is injective. Then for any positive integer $\alpha$,
the Finsler (pseudo)metric $F_{\alpha}$ constructed above satisfies the following curvature inequality
\[
\d\d^c \log \lvert \gamma'(z) \lvert^2_{F_{\alpha}} \gtrsim \gamma^*\omega_{FS}.
\]  
\end{thm} 
\begin{proof}
By the Poincar\'e-Lelong formula, we know that
$\d\d^c \log \lvert \gamma'(z) \lvert^2_{F_{\alpha}} = - \Theta(T_C, (\mathrm{d}\gamma)^*F_{\alpha}) + R \geq - \Theta(T_C, (\mathrm{d}\gamma)^*F_{\alpha})$, where $R$ is the ramification divisor of $\gamma$. Denote by $N$ the saturation of the image of $\d\gamma :\, T_C \to \gamma^*T_Y(-\log S)$. Then we have the following curvature current estimate
\begin{align}\label{est1} 
\Theta(T_C, (\mathrm{d}\gamma)^*F_{\alpha}) &\leq \Theta(N,\gamma^*F_{\alpha}) = \frac{1}{m} \Theta(N^{\otimes m},\gamma^*F^{\otimes m}_{\alpha})\\\nonumber
&\leq \frac{1}{m} \gamma^*\Theta(Sym^m T_Y(-\log S), (\rho^{n-m,m} \circ \tau^m)^*(g^{-1}_{\alpha} \otimes g_{hod}))|_{N^{\otimes m}}   \\\nonumber
&\leq \frac{1}{m} \gamma^* \Theta(A^{-1} \otimes E, g^{-1}_{\alpha} \otimes g_{hod})|_{\gamma^*(\rho^{n-m,m} \circ \tau^m)(N^{\otimes m})}\\\nonumber
&= -\frac{1}{m} \gamma^*\Theta(A,g_{\alpha}) + \frac{1}{m} \gamma^*\Theta(E,g_{hod})|_{\gamma^*A \otimes \gamma^*(\rho^{n-m,m} \circ \tau^m)(N^{\otimes m})}. 
\end{align}
Recall that $m$ is the maximal length of the iteration along $\gamma$ so that $\gamma^*A \otimes \gamma^*(\rho^{n-m,m} \circ \tau^m)(N^{\otimes m})$ lies in the kernel of $\theta_{\gamma}$. Therefore, as the last term in %
the estimate (\ref{est1}) 
is semi-negative by the Griffiths curvature computation (see {\cite[Lemma~(7.18)]{Sch-73}} or \cite{Zuo-00}), we have as currents that
\[
\Theta(T_C, (\mathrm{d}\gamma)^*F_{\alpha}) \leq -\frac{1}{m} \gamma^*\Theta(A,g_{\alpha}).
\]
And hence, we have
\[
\d\d^c \log \lvert \gamma'(z) \lvert^2_{F_{\alpha}} \geq \frac{1}{m} \gamma^*\Theta(A,g_{\alpha}) \gtrsim \gamma^*\omega_{FS},
\]
the last inequality being given by Lemma~\ref{curv-estimate}
\end{proof}
Using the curvature inequality in Theorem~\ref{curv-ineq-for-criteria} and the estimate in Proposition~\ref{abstract-curv-bound}, Viehweg and the fourth named author  
showed that $F_{\alpha}$ is strongly negatively curved along the analytic curve.%
\begin{thm}[Viehweg-Zuo]\label{curv-bound}
Keeping the assumptions on the family $V \to U$ the same as in Theorem~\ref{deng}.
Fix the analytic map $\gamma:\, C \to U$. Then there exists a positive integer $\alpha$ (depending on the maximal length along $\gamma$) and  $c_{\alpha}>0$ (depending on $\alpha$) such that the curvature of $F_\alpha$ satisfies
\[
K_{F_{\alpha}}(v) \leq -c_{\alpha}
\]
for any nonzero tangent vector $v:= \gamma'(z)$ of $\gamma$.
\end{thm}

\begin{rmk}
(1) Although $C=\mathbb{C}$ in \cite{VZ-1}, all the arguments above work for a general Riemann surface $C$, except the final step in {\cite[Lemma~7.9]{VZ-1}} where the Ahlfors-Schwarz lemma is used.  \\
(2) Since $(\mathrm{d}\gamma)^*F_{\alpha}$ is locally bounded on $T_C$ by construction, the inequality above entails an inequality in the sense of currents just as that in Theorem~\ref{abstract-curv-bound}.
\end{rmk}

Before entering the next section, we list two crucial points to the the arguments we present there: 
\begin{itemize}
\item \emph{logarithmic growth of the Hodge metric near boundary:} 
In fact this is the crucial point of Viehweg-Zuo's curvature estimates; and those estimates are crucial to our argument.\\[-4mm]
\item \emph{local boundedness of the Finsler pseudometric}: 
Used in defining the order function $T(r)$.
\end{itemize}

\section{Big Picard theorem via negative curvature}\label{sect_BP}

In this section we shall prove Theorem~\ref{thm:Big Picard} and Theorem~\ref{big-picard new} by using the negative curvature method  inspired by the argument in \S9 of Griffiths-King \cite{Gri-King-73}. Now, we have two negatively curved Finsler pseudometrics: $h_{\alpha}$ of Proposition~\ref{abstract-curv-bound} and $F_{\alpha}$ of Theorem~\ref{curv-bound}. Since $h_{\alpha}$ shares the same curvature properties as $F_{\alpha}$ thanks to Proposition~\ref{abstract-curv-bound},  we only present the proof of Theorem~\ref{big-picard new} using $F_{\alpha}$. The proof of Theorem~\ref{thm:Big Picard} using $h_{\alpha}$ is verbatim.\\[-1mm]

We identify $\mathbb{D}^*$ with the inverted punctured unit disk $\mathbb{D}^\circ:=\{ z \in \mathbb{C}; \text{  $|z| > 1$}\}$ in order to match the usual notations in Nevanlinna theory for entire curves. We set $\mathbb{D}_{r_0,r}:= \{ r_0 \leq |z| <r\}\subset \mathbb{C}^*$  and $\mathbb{D}_r:=\{ z \in \mathbb{C}; |z| <r\}$. Denote by $\gamma:\, \mathbb{D}^\circ \to U$ the analytic map in question. Then we want to show that $\gamma$ extends over the point at infinity. We fix an $r_0>1$ from now on.\\[-2mm]

By the constructions in Subsection~\ref{sec_finsler-metric}, there is a Finsler pseudometric $F_{\alpha}$ on $T_Y(-\log S)$ (resp. $h_{\alpha}$ on $T_X(-\log D)$), or equivalently a semi-norm on the tautological line bundle $\mathcal{O}_{{\mathbb P}_{\text{alg}}(T_Y(-\log S)^\vee)}(-1)$, 
with the following inequalities of curvature currents
\begin{align*} 
\gamma^*\omega_{FS} \lesssim dd^c \log ({\mathbb P}\gamma')^{-1}\lvert\ \lvert^2_{F_{\alpha}} \\
  \omega_{\gamma} \lesssim dd^c \log ({\mathbb P}\gamma')^{-1}\lvert\ \lvert^2_{F_{\alpha}} ,
\end{align*}
by Theorems~\ref{curv-ineq-for-criteria} and \ref{curv-bound} (or by the hypothesis in Theorem~\ref{thm:Big Picard} and Proposition~\ref{abstract-curv-bound}) respectively. Here $\omega_{\gamma}$ is the semi-positve (1,1)-form associated to the semi-norm $({\mathbb P}\gamma')^{-1}(\lvert\ \lvert_{F_{\alpha}}\circ \gamma_*)$ on $T_{\mathbb{D}^\circ}$.  

\begin{rmk}\label{rem:proof}
For the construction of $F_{\alpha}$, remember that one needs to change the birational model of $U$ in the construction of those two Higgs bundles $(F,\tau)$ and $(E,\theta)$. In our application, we can always assume that the image of $\gamma$ is Zariski dense by replacing $\bar{Y}$ by the Zariski closure of $\gamma(\mathbb{D}^\circ)$. Then the analytic map lifts to $\tilde{\gamma}:\, \mathbb{D}^\circ \to \tilde{U}$, where $\tilde{U}$ is the new birational model for the new zariski closure base space. Clearly, it suffices to prove the extension property for $\tilde{\gamma}$.  
\end{rmk}

By the above argument, the following Nevanlinna characteristic {(or order)} functions 
\begin{equation} \label{eq:order}
T_{\gamma^*\omega_{FS}}(r) := \int^r_{r_0} \frac{d\rho}{\rho} \int_{\mathbb{D}_{r_0,\rho}} \gamma^*\omega_{FS}
\end{equation}
\begin{equation*} 
T_{\omega_{\gamma}} (r) := \int^r_{r_0} \frac{d\rho}{\rho} \int_{\mathbb{D}_{r_0,\rho}} \omega_{\gamma}
\end{equation*}  
are both dominated (i.e. $\lesssim$) by 
\begin{align}\label{eq:dominant}
\int^r_{r_0} \frac{d\rho}{\rho} \int_{\mathbb{D}_{r_0,\rho}} dd^c \log ({\mathbb P}\gamma')^{-1}\lvert\ \lvert^2_{F_{\alpha}} \le  \int^r_{r_0} \frac{d\rho}{\rho} \int_{\mathbb{D}_{r_0,\rho}} dd^c \log \lvert \gamma'(z) \lvert^2_{F_{\alpha}} 
\end{align}

It is elementary and classical that the asymptotic behavior of $T_{\gamma^*\omega_{FS}}(r)$ as $r\to \infty$ characterizes whether $\gamma$ can be extended over $\infty$  (see \emph{e.g.} \cite[2.11.~cas local]{Dem97b} or \cite[Remark 4.7.4.(\lowerromannumeral{2})]{NW}).  
 \begin{lem}\label{lem:criteria}
	$T_{\gamma^*\omega_{FS}}(r)=O(\log r)$ if and only if $\gamma$ can be  extended holomorphically over $\infty$. \qed
\end{lem}
We will need the following Green-Jensen formula for the punctured disk.
\begin{prop}[Green-Jensen formula on punctured disk]
Let $\phi$ be function on $\mathbb{D}^\circ$ such that $\phi$ is differentiable outside a discrete set of points disjoint from ${\mathbb D}_{r_0}$ and $dd^c\phi$ exists as a current. Then for $0<r_0<r$, 
\begin{align}\label{eq:Jensen}
\int^r_{r_0} \frac{d\rho}{\rho} \int_{\mathbb{D}_{r_0,\rho}} dd^c\phi =  \int_{\partial\mathbb{D}_r} \phi \cdot d^c\log |z| - \int_{\partial\mathbb{D}_{r_0}} \phi \cdot d^c\log |z| - (\log \frac r{r_0})\cdot \int_{\partial\mathbb{D}_{r_0}} d^c\phi.  \ \ \ \ \ \ \ \ \ \ \qed
\end{align}
\end{prop}
\vspace{2mm}
By using \eqref{eq:Jensen}, we have for {a fixed} $r_0>1$ that, 
\begin{equation}
  \label{eq:2}
\int^r_{r_0} \frac{d\rho}{\rho} \int_{\mathbb{D}_{r_0,\rho}} dd^c \log \lvert \gamma'(z) \lvert^2_{F_{\alpha}}
 = \int_{\partial\mathbb{D}_r} \log \lvert \gamma'(z) \lvert^2_{F_{\alpha}} \cdot d^c \log |z| + O(\log r).
\end{equation}
We define the first term of the right hand side of (\ref{eq:2}) to be the \emph{modified proximity function} $m_{\omega_{\gamma}}(r)$.
%
\begin{lem}\label{log-derivitive}
Denote by $\frac{d}{d s} := r \cdot \frac{d}{d r}$ the logarithmic derivative. Then
\[
m_{\omega_{\gamma}}(r) \leq \log \frac{d^2 T_{\omega_{\gamma}} (r)}{d s^2} + O(\log r).
\]  
\end{lem}
\begin{proof}
Note that we can write $\omega_{\gamma} = \lvert \gamma'(z) \lvert^2_{F_{\alpha}} \cdot  \frac{\sqrt{-1}}{2\pi} dz \wedge d\bar{z}$.
Denote by $\xi := \lvert \gamma'(z) \lvert^2_{F_{\alpha}}$ for simplicity.
By direct computation one finds that
\[
\frac{1}{r} \, \frac{d}{dr} \left( r \frac{d}{dr} T_{\omega_{\gamma}} (r) \right) =  \int_{\partial \mathbb{D}_r} \xi \cdot d^c \log |z|.
\]
Using the concavity of the logarithmic function, we obtain
\begin{align*} 
m_{\omega_{\gamma}} (r) &= \int_{\partial\mathbb{D}_r} \log \xi \cdot d^c \log |z|
 \leq \log \int_{\partial \mathbb{D}_r} \xi \cdot d^c \log |z| \\ 
&= \log \left\{ \frac{1}{r} \, \frac{d}{dr} \left( r \frac{d}{dr} T_{\omega_{\gamma}} (r) \right) \right\} 
 = \log \left\{ r^{-2} \cdot \frac{d^2}{ds^2} T_{\omega_{\gamma}} (r) \right\} \\ &= -2 \, \log r + \log \frac{d^2 T_{\omega_{\gamma}} (r)}{d s^2}. 
\end{align*}  \end{proof}
\vspace{-3mm}
Applying the Calculus lemma twice, we obtain for some $ \epsilon,\delta>0$ that $$\frac{d^2 T_{\omega_{\gamma}} (r)}{d s^2} \leq r^{2+\epsilon} \cdot T_{\omega_{\gamma}} (r)^{2+\delta} \,\,\, ||,$$
where $||$ here means that the inequality holds for $r$ outside a set of finite Lebesque measure. Thus 
\[
m_{\omega_{\gamma}}(r) \leq (2+\delta)\cdot \log T_{\omega_{\gamma}} (r) + O(\log r) \, ||.
\]
Combining \eqref{eq:dominant} and \eqref{eq:2}, we get the inequalities
\begin{align*} 
T_{\omega_{\gamma}} (r) &\lesssim m_{\omega_{\gamma}}(r) + O(\log r) \lesssim  \log T_{\omega_{\gamma}} (r) + O(\log r) \, ||; \\
T_{\gamma^*\omega_{FS}} (r) &\lesssim m_{\omega_{\gamma}}(r) + O(\log r) \lesssim  \log T_{\omega_{\gamma}} (r) + O(\log r) \, ||. 
\end{align*}
The first implies that $T_{\omega_{\gamma}} (r) = O(\log r)$, which then combines with the second to yield
\[
T_{\gamma^*\omega_{FS}}(r) = O(\log r). 
\]
By Lemma~\ref{lem:criteria}, $\gamma$ extends holomorphically over infinity and completes our proof of Theorem~\ref{thm:Big Picard}.

\begin{rmk}\label{relaxation}{
One can see that the assumption on the family in Theorem~\ref{big-picard new} can be relaxed to a family of polarized manifolds with maximal variation of moduli if we require the image of the holomorphic map $\gamma:\, \mathbb{D}^* \to U$ to be not contained in a proper subvariety of $U$. We choose such a family whose base is nonsingular and a proper subvariety $Z$ so that the map $\tau^1$ defined in subsection~\ref{inj-comparison} is injective outside $Z$. Since our $\gamma(\mathbb{D}^*)$, being zariski dense, is not contained in $Z$, the arguments in sections 2 and 3 go through, giving the extension of $\gamma$; see out setup just before Section~\ref{comparison1}. This completes the proof Theorem~\ref{big-picard new}. }
\end{rmk}

\begin{rmk} It is a classical fact{, up to adding the term $O(\log r)$, that} the charateristic function $$T_{({\mathbb P}\gamma')^*{\mathcal O}(1)} :=\int^r_2 \frac{d\rho}{\rho} \int_{\mathbb{D}^\circ_\rho} dd^c \log ({\mathbb P}\gamma')^{-1}\lvert\ \lvert^2_{h_{\alpha}}$$ is independent of a \emph{non-degenerate} Finsler metric $h_{\alpha}$ 
chosen on ${\mathcal O}(-1)=\mathcal{O}_{{\mathbb P}_{\text{alg}}(T_Y(-\log S)^\vee)}(-1)$.
So the estimate above also gives us the so called tautological inequality in our case in the form:
\begin{align}\label{taut}
T_{({\mathbb P}\gamma')^*{\mathcal O}(1)}(r) = O(\log r).
\end{align}
\noindent
In fact, the tautological inequality {was first established by McQuillan for entire curves and requires no hypothesis on curvature, albeit with extra terms on the right side of (\ref{taut}).
It generalizes easily to the case of punctured disks and we state a version where the holomorphic curve lies outside the boundary divisor (cf. \cite[\S 29]{Vojta}, which has a nice proof more amenable to arithmetic geometry).}
{\begin{thm}[Tautological Inequality for the punctured disk]\label{Tauto} 
Let $X$ be a projective manifold with a Fubini-Study metric $\omega_{FS}$ and let $D$ be a simple normal crossing divisor on $X$. 
Let $\gamma:\mathbb{D}^\circ\to X\setminus D$  be a  holomorphic map. Then $\displaystyle T_{({\mathbb P}\gamma')^*O(1)}(r) 
\le O(\log^+ T_{\gamma^*\omega_{FS}}(r)) + O(\log r)\ \Vert.\ \ \ \ \ \ \ \ \ \ \ \ \square$
\end{thm}}

Our proof of Theorem~\ref{thm:Big Picard} in the next section also  {proves this inequality, but more geometrically.}
Note that Theorem~\ref{thm:Big Picard} derives easily from this ineqality and (\ref{taut}) by choosing a hermitian metric $h$ dominating the possibly degenerate $h_\alpha$. (The proximity function with respect to the former then dominates that of the latter and so, up to an $O(\log r)$ term, so does the characteristic functions.)
\end{rmk}

\section{Big Picard theorem via the lemma on the logarthmic derivative}\label{sect_BP2}

In a similar vein as that in establishing the fundamental vanishing theorem for (symmetric or jet) differentials, pioneered by Green-Griffiths \cite{GG79} and completed by Siu-Yeung \cite{S-Y97} and Demailly \cite{Dem97b} {via the logarithmic derivative lemma, we give another proof of Theorem~\ref{thm:Big Picard} in this section. {Theorem~\ref{big-picard new} is then an immediate corollary given our setup, as shown in Remark~\ref{relaxation}.}}

\subsection{Preliminary in Nevannlina theory} Let $(X,\omega)$ be a compact K\"ahler manifold. Consider a holomorphic map $\gamma:\mathbb{D}^*\to X$. We identify $\mathbb{D}^*$ with $\mathbb{D}^\circ:=\{ z \in \mathbb{C}; \text{  $|z| > 1$}\}$ via $z\mapsto \frac{1}{z}$ as before. We can extend $\gamma$ to a $\gamma: \mathbb{D}^*\cup \mathbb{D}^\circ \to X$ by setting $\gamma(z)=\gamma(\frac 1{z})$ for $z\in\mathbb{D}^\circ$. 

Fix any $r_0>1$. Recall that the \emph{characteristic function} in \eqref{eq:order} is defined by 
$$
 T_{\gamma^*\omega}(r):= \int_{r_0}^{r}\frac{d\rho}{\rho}\int_{\bD_{r_0,\rho}}\gamma^*\omega.
$$  
Let us first state a couple of useful inequalities.
\begin{lem}
	Write $\log^+x:={\rm max}(\log x ,0)$. Then
	\begin{align}\label{eq:concave}
	\log^+(\sum_{i=1}^{N}x_i)\leq \sum_{i=1}^{N}\log^+x_i+ \log N,&\quad  \log^+\prod_{i=1}^{N}x_i\leq \sum_{i=1}^{N}\log^+x_i\quad\mbox{ for }x_i\geq 0. 
	\end{align}  \\[-8mm]\qed
\end{lem}

The following lemma is well-known (see \emph{e.g.} \cite[Lemme 1.6]{Dem97b}).
\begin{lem}\label{lem:decrease}
	Let $ X$ be a projective manifold equipped with a hermitian metric $\omega$ and let $u:X\to \mathbb{P}^1$ be a rational function. Then for any holomorphic map $\gamma:\mathbb{D}^\circ\to X$, one has
	\begin{align}\label{eq:decrease}
	T_{(u\circ\gamma)^*\omega_{FS}}(r)\leq CT_{\gamma^*\omega}(r)+O(1)
	\end{align} 
	where $\omega_{FS}$ is the Fubini-Study metric for $\mathbb{P}^1$. \qed
\end{lem}

 The following logarithmic derivative lemma for the punctured disk, see e.g.~\cite{Nog81}[Lemma~2.12],   is crucial in our proof. 
\begin{lem}\label{lem:Noguchi}
	Let $u:\mathbb{D}^\circ\to \mathbb{P}^1$ be any meromorphic function. Then for any $k\geq 1$, we have
	\begin{align}\label{eq:Noguchi}
	\frac{1}{2\pi}\int_{0}^{2\pi}\log ^+|\frac{u^{(1)}(re^{i\theta})}{u(re^{i\theta})}|d\theta \leq C(\log^+ T_{u^*\omega_{FS}}(r)+\log r)   \quad \lVert, 
	\end{align}
	for some constant $C>0$ which does not depend on $r$. 
	Here   the symbol \(\lVert\) means that	the  inequality holds outside a Borel subset of \((r_0,+\infty )\) of finite Lebesgue measure.  \qed
\end{lem}

We also need an elementary lemma (due to E. Borel), called the calculus lemma.
\begin{lem}\label{lem:Borel}
	Let $\phi(r)\geq 0 (r\geq r_0\geq 0)$ be a monotone increasing
	function. For every $\delta>0$, 
	\begin{align}\label{eq:Borel}
	\frac{d}{d r}\phi(r)\leq \phi(r)^{1+\delta}\quad \lVert.
	\end{align}  \\[-8mm]\qed
\end{lem}

\subsection{Criterion for big Picard   theorem} 
 Now we are ready to give a new proof of Theorem~\ref{thm:Big Picard}.
\begin{proof} Via the isomorphism $\mathbb{D}^\circ\stackrel{\sim}{\rightarrow} \mathbb{D}^*$ by setting $z\mapsto \frac{1}{z}$ as before, we  assume that $\gamma:\mathbb{D}^*\to X-D$ is a holomorphic map from $\mathbb{D}^\circ$ to $X-D$.
	We take a finite affine covering $\{U_\alpha\}_{\alpha\in I}$ of $X$ and rational  functions
	$(x_{\alpha1},\ldots,x_{\alpha n})$ on $X$ which are holomorphic on $U_\alpha$ so that
	\begin{align*}
	dx_{\alpha1}\wedge\cdots\wedge dx_{\alpha n}\neq 0 \ \mbox{ on } U_\alpha\\
	D\cap U_\alpha=(x_{\alpha,s(\alpha)+1} \cdots  x_{\alpha n}=0)
	\end{align*}
	Hence 
	\begin{align}\label{eq:basis}
	(e_{\alpha 1},\ldots,e_{\alpha n}):=	(\frac{\partial }{\partial x_{\alpha 1}},\ldots,\frac{\partial }{\partial x_{\alpha s(\alpha)}},x_{\alpha, s(\alpha)+1}\frac{\partial }{\partial x_{\alpha, s(\alpha)+1}},\ldots,x_{\alpha n}\frac{\partial}{\partial x_{\alpha n}})
	\end{align} 
	is a basis for $T_X(-\log D)|_{U_\alpha}$. Write $$ (\gamma_{\alpha1}(z),\ldots,\gamma_{\alpha n}(z)):=(x_{\alpha 1}\circ \gamma,\ldots,x_{\alpha n}\circ \gamma) $$  
	so that $\gamma_{\alpha j}: \mathbb{D}^\circ\to \mathbb{P}^1$ is a meromorphic function  over $\mathbb{D}^\circ$ for any $\alpha$ and $j$.  With respect to the trivialization of $T_X(-\log D)$ induced by the basis \eqref{eq:basis},   $\gamma'(z)$ can be written as
	$$\gamma'(z)= \gamma_{\alpha1}'(z)e_{\alpha 1}+\cdots+\gamma_{\alpha s(\alpha)}'(z)e_{\alpha s(\alpha)}+(\log \gamma_{\alpha,s(\alpha)+1})'(z)e_{\alpha, s(\alpha)+1}+\cdots+(\log \gamma_{\alpha  n})'(z)e_{\alpha n} $$
	over $U_\alpha$. Let $\{\rho_\alpha \}_{\alpha\in I}$ be a partition of unity subordinated to $\{U_\alpha\}_{\alpha\in I}$. 
	
	 Since $h$ is Finsler pseudometric for $T_X(-\log D)$ which is continuous and locally bounded from above by Definition \ref{def:Finsler}, and $ I$ is a finite set,    there is a constant $C>0$ so that
		\begin{align}\label{eq:unity}
		\rho_\alpha\circ \gamma \cdot|\gamma'(z)|_h^2\leq C \Big( \sum_{j=1}^{s(\alpha)}\rho_\alpha\circ \gamma \cdot|\gamma_{\alpha j}'(z)|^2+\sum_{i=s(\alpha)+1}^{n}|(\log \gamma_{\alpha i})'(z)|^2 \Big) \quad \forall z\in \mathbb{D}^*
		\end{align}
		for any $\alpha$. Hence
		\begin{align}\nonumber
		T_{\gamma^*\omega}(r)&:=\int_{r_0}^{r}\frac{d\rho}{\rho}\int_{\mathbb{D}_{r_0,\rho}}\gamma^*\omega\stackrel{\eqref{eq:condition}}{\leq} \int_{r_0}^{r}\frac{d\rho}{\rho}\int_{\mathbb{D}_{r_0,\rho}}	 \hess \log |\gamma'|_h^2\\\nonumber
		&\stackrel{\eqref{eq:Jensen}}{\leq}  \frac{1}{2\pi}\int_{0}^{2\pi} \log|\gamma'(re^{i\theta})|_h d\theta+O(\log r)\\\nonumber
		&\leq  \frac{1}{2\pi}\int_{0}^{2\pi} \log^+\sum_{\alpha}|\rho_\alpha\circ \gamma\cdot \gamma'(re^{i\theta})|_h d\theta+O(\log r)\\\nonumber
		&\stackrel{\eqref{eq:concave}}{\leq} \sum_{\alpha}\frac{1}{2\pi}\int_{0}^{2\pi} \log^+|\rho_\alpha\circ \gamma\cdot \gamma'(re^{i\theta})|_h d\theta+O(\log r)\\\nonumber
		&\stackrel{\eqref{eq:unity}+\eqref{eq:concave}}{\leq} \sum_{\alpha}\sum_{i=s(\alpha)+1}^n \frac{1}{2\pi}\int_{0}^{2\pi} \log^+|(\log \gamma_{\alpha i})'(re^{i\theta})| d\theta\\\nonumber
		&\quad +\sum_{\alpha}\sum_{j=1}^{s(\alpha)}\frac{1}{2\pi}\int_{0}^{2\pi}\log ^+|\rho_\alpha\circ \gamma\cdot \gamma_{\alpha j}'(re^{i \theta})|d\theta +O(\log r)\\\nonumber
		&\stackrel{\eqref{eq:Noguchi}}{\leq} C_1\sum_{\alpha}\sum_{i=s(\alpha)+1}^{n}\big(\log^+  T_{\gamma_{\alpha i},\omega_{FS}}(r) +\log r\big)\\\nonumber
		&\quad +\sum_{\alpha}\sum_{j=1}^{s(\alpha)}\frac{1}{2\pi}\int_{0}^{2\pi}\log ^+|\rho_\alpha\circ \gamma\cdot \gamma_{\alpha j}'(re^{i \theta})|d\theta+O(\log r)\quad \lVert \\\label{eq:final}
		&\stackrel{\eqref{eq:decrease}}{\leq} C_2(\log^+T_{\gamma^*\omega}(r)+ \log r)+\sum_{\alpha}\sum_{j=1}^{s(\alpha)}\frac{1}{2\pi}\int_{0}^{2\pi}\log ^+|\rho_\alpha\circ \gamma\cdot \gamma_{\alpha j}'(re^{i \theta})|d\theta\quad \lVert.
		\end{align}
		Here $C_1$ and $C_2$ are two positive constants which do not depend on $r$.
		\begin{claim}\label{claim}
			For any $\alpha\in I$, any $j\in \{1,\ldots,s(\alpha)\}$, one has
			\begin{align}\label{eq:new}
			\frac{1}{2\pi}\int_{0}^{2\pi}\log ^+|\rho_\alpha\circ \gamma\cdot \gamma_{\alpha j}'(re^{i \theta})|d\theta\leq C_3(\log^+T_{\gamma^*\omega}(r)+ \log r)+O(1)\quad \lVert
			\end{align} 
			for a	positive constant $C_3$ which does not depend on $r$.
		\end{claim} 
		\begin{proof}[Proof of Claim \ref{claim}] \
			The proof of the claim is borrowed from \cite[eq.(4.7.2)]{NW}. Pick $C>0$ so that $\rho^2_\alpha\sn dx_{\alpha j}\wedge d\bar{x}_{\alpha j}\leq C\omega$. Write $\gamma^*\omega:=\sn B(t)dt\wedge d\bar{t}$. Then
			\begin{align*}
			&\frac{1}{2\pi}\int_{0}^{2\pi}\log ^+|\rho_\alpha\circ \gamma\cdot \gamma_{\alpha j}'(re^{i \theta})|d\theta = \frac{1}{4\pi}\int_{0}^{2\pi}\log ^+(|\rho^2_\alpha\circ \gamma |\cdot|\gamma_{\alpha j}'(re^{i \theta})|^2)d\theta\\
			&\leq \frac{1}{4\pi}\int_{0}^{2\pi}\log ^+B(re^{i\theta})d\theta+O(1) \leq \frac{1}{4\pi}\int_{0}^{2\pi}\log (1+B(re^{i\theta}))d\theta+O(1)\\
			&\leq \frac{1}{2}\log (1+\frac{1}{2\pi}\int_{0}^{2\pi}   B(re^{i\theta})d\theta)+O(1) 
			= \frac{1}{2}\log (1+\frac{1}{2\pi r}\frac{d}{d r}\int_{\mathbb{D}_{r_0,r}}   rBdrd\theta)+O(1)\\
			&= \frac{1}{2}\log (1+\frac{1}{2\pi r}\frac{d}{d r} \int_{\mathbb{D}_{r_0,r}}\gamma^*\omega)+O(1)\\
			&\stackrel{\eqref{eq:Borel}}{\leq}  \frac{1}{2}\log (1+\frac{1}{2\pi r} (\int_{\mathbb{D}_{r_0,r}}\gamma^*\omega)^{1+\delta})+O(1)\quad \lVert\\
			&=  \frac{1}{2}\log (1+\frac{r^\delta}{2\pi } (\frac{d}{d r} T_{\gamma^*\omega}(r))^{1+\delta})+O(1)\quad \lVert\\ 
			&\stackrel{\eqref{eq:Borel}}{\leq}   \frac{1}{2}\log (1+\frac{r^\delta}{2\pi } ( T_{\gamma^*\omega}(r))^{(1+\delta)^2})+O(1)\quad \lVert\\
			&\leq 4\log^+T_{\gamma^*\omega}(r)+\delta\log r+O(1)\quad \lVert.
			\end{align*}
			Here we pick $0<\delta<1$ and the last inequality follows.
			The claim is proved.
		\end{proof}
		Putting \eqref{eq:new}   to \eqref{eq:final},  one obtains
		$$
		T_{\gamma^*\omega}(r)\leq C(\log^+T_{\gamma^*\omega}(r)+ \log r)+O(1)\quad \lVert 
		$$
		for some positive constant $C$. 
		Hence $T_{\gamma^*\omega}(r)=O(\log r)$.  We  apply Lemma \ref{lem:criteria} to conclude that $\gamma$ extends to the $\infty$.  
\end{proof} 

%

 \section{Algebraic hyperbolicity for moduli spaces of polarized manifolds}\label{alg_hyp}
 In this section we will prove {the} algebraic hyperbolicity of {the} moduli {spaces studied above.} The question was {posed} to the first named author by Erwan Rousseau in February 2019 at {the} CIRM.
\begin{proof}[Proof of Theorem \ref{thm:algebraic hyperbolicity}]
{ Since the moduli map of the polarized family $(f:V\to U,\mathcal{L})\in \mathcal{M}_h(U)$ is quasi-finite, the family $f$ is of maximal variation and we let  $U^\circ\subset U$ be the Zariski open set given {in} Theorem \ref{deng}.	We first take a projective compactification $Z_0$ of the base $U$ and let $D=Z_0 - U$.  
Let $Z_1,\ldots,Z_m$ be the irreducible components of $Z_0-U^\circ$ whose intersection{s with $U$ are} nonempty. Let $\mu_i:X_i\to Z_i$ be a desingularization so that $D_i:=\mu_i^{-1}(D)$ is of simple normal crossing. For each 
$U_i:=X_i-D_i$, as the moduli map of the new family $(f_i:V\times_{U}U_i\to U_i, \mathcal{L}|_{V\times_{U}U_i})$ is generically finite, $f_i$ is also of maximal variation and so yields a Zariski open subset $U_i^\circ\subset U_i$ by Theorem \ref{deng} that allows us to repeat for $i\ge 1$ the construction.  Iterating using Theorem \ref{deng} and applying Theorem \ref{curv-ineq-for-criteria},  we construct a set of  log pairs $\{ (X_j,D_j)\}_{j=0,\ldots,N} $ with the following properties. }
	\begin{enumerate}
	\item There are   morphisms $\mu_i:X_i\to Y$ with $\mu_i^{-1}(D)=D_i$, so that each $\mu_i:X_i\to \mu_i(X_i)$ is a birational morphism.
	\item There are smooth Finsler pseudometrics $h_{i1},\ldots,h_{in}$ for $T_{X_i}(-\log D_i)$.
	\item \label{iso}$\mu_i|_{U^\circ_i}:U^\circ_i\to \mu_i(U^\circ_i)$ is an isomorphism.
	\item \label{curvature} There are   smooth K\"ahler metrics $\omega_{i1},\ldots,\omega_{in}$ on $X_i$ such that for any curve $\gamma:C\to U_i$ with $C$ an open set of $\mathbb{C}$ and $\gamma(C)\cap U^\circ_i\neq 0$, there exists some $h_{ij}$ so that $|\gamma'(t)|^2_{h_{ij}}\not\equiv 0$, and 
	\begin{align}\label{eq:uniform}
	\hess \log |\gamma'|_{h_{ij}}^2\geq \gamma^*\omega_{ij}.   
	\end{align}
	\item   For any $i\in \{0,\ldots,N\}$, either $		\mu_i( U_i)-\mu_i(U_i^\circ)$ is zero dimensional, or  there exists $I\subset  \{0,\ldots,N\}$ so that 
	$$
	\mu_i( U_i)-\mu_i(U_i^\circ)\subset \cup_{j\in I}\mu_j(X_j)
	$$
\end{enumerate} 	 

For any irreducible and reduced curve $C\subset Y$ with $C\not\subset D$. By the above construction, there is some log $(X_i,D_i)$  so that $C\subset \mu_i(X_i)$ and $C\cap \mu_i(U_i^\circ)\neq \varnothing$. By Item \ref{iso}, $C$ is not contained in the exceptional locus of $\mu_i$, and let $C_i\subset X_i$ be the strict transform of $C$ under $\mu_i$.
	Denote by $\nu_i:\tilde{C}_i\to C_i\subset X_i$
	the normalization of $C_i$, and set $P_i:=(\mu_i\circ\nu_i)^{-1}(D)=\mu_i^{-1}(D_i)$. Then  one has
	$$
	d \nu_i:T_{\tilde{C}_i}(-\log P_i)\to \nu_i^*T_{X_i}(-\log D_i).
	$$
	By Item \ref{curvature}, there is a Finsler pseudometric $h_{ij}$ for $T_{X_i}(-\log D_i)$ so that \eqref{eq:uniform} holds.  Consider
 $\tilde{h}_i:=\nu_i^*h_{ij}$, which is a complex semi-norm over $T_{\tilde{C}_i}(-\log P_i)$. By \eqref{eq:uniform}, there is a K\"ahler metric $\omega_{ij}$ on $X_i$ so that the  curvature current satisfies
	$$
	\frac{\sn}{2\pi}\Theta_{\tilde{h}_i^{-1}}(K_{\tilde{C}_i}(\log P_i))\geq \nu_i^*\omega_{ij}
	$$
Since $\mu_i\circ\nu_i:\tilde{C}_i\to C$ is the normalization of $C$, one has
	$$
	2g(\tilde{C}_i)-2+i(C,D)=\int_{\tilde{C}_i}  \frac{\sn}{2\pi}\Theta_{\tilde{h}_i^{-1}}(K_{\tilde{C}_i}(\log P_i))\geq \int_{\tilde{C}_i} \nu_i^*\omega_{ij}
	$$
	Fix a K\"ahler metric $\omega_Y$ on $Y$. Then there is a constant $\ep>0$ so that $\omega_{ij}\geq\ep\mu_i^*\omega_Y$ for any $i=0,\ldots,N$ and $j=1,\ldots,n$. We thus have
	$$
	2g(\tilde{C}_i)-2+i(C,D)  \geq \ep\int_{\tilde{C}_i} \mu_i^*\omega_Y=\ep\deg_{\omega_Y}C
	$$
 This shows the algebraic hyperbolicity of the base $U$.
\end{proof}
\begin{rmk}\label{rem:Arakelov}
	The algebraic hyperbolicity in Theorem \ref{thm:algebraic hyperbolicity} generalizes the Arakelov-type inequalities in \cite{MVZ-06} by M\"oller, Viehweg and the fourth named author, as well as the  weak boundedness of moduli stacks of canonically polarized manifolds in \cite{KL} by Kov\'acs-Lieblich. In \cite[Theorem 0.3]{MVZ-06}, the authors obtained   Arakelov-type inequalities  with \emph{sharp bounds} for semistable families of projective manifolds with semi-ample canonical sheaf over $\mathbb{P}^1$.  In \cite[Definition 2.4]{KL}, the authors introduced the notion of \emph{weak boundedness} for quasi-projective varieties (which is weaker than the notion of algebraic hyperbolicity) and they proved that the moduli stacks of canonically polarized manifolds are weakly bounded.
	\end{rmk}

\end{document}